\documentclass[12pt,sumlimits,intlimits]{amsart}
\usepackage{url,cite,hyperref,amsmath,amsthm,amssymb,mathtools}
\usepackage[margin=1.25in]{geometry}
\usepackage[all]{xy}
\usepackage{comment}
\usepackage{xcolor}

\newtheorem{theorem}{Theorem}[section]
\newtheorem{proposition}[theorem]{Proposition}
\newtheorem{corollary}[theorem]{Corollary}
\newtheorem{lemma}[theorem]{Lemma}

\theoremstyle{definition}
\newtheorem{definition}[theorem]{Definition}
\newtheorem{remark}[theorem]{Remark}
\newtheorem{example}[theorem]{Example}
\DeclareMathOperator{\GL}{GL}
\newcommand{\la}{\langle}
\newcommand{\ra}{\rangle}

\newcommand{\Z}{\mathbb{Z}}
\newcommand{\Q}{\mathbb{Q}}

\newcommand{\HH}{\mathbb{H}}

\newcommand{\R}{\mathbb{R}}
\newcommand{\T}{\mathbb{T}}

\newcommand{\Aa}{A_\theta\otimes A_\theta}
\numberwithin{equation}{section}

\title[Unitriangular groups]{Classification of C*-algebras generated by representations of the unitriangular group $UT(4,\Z)$}
\author{Caleb Eckhardt}
\author{Craig Kleski}
\author{Paul McKenney}

\address{Department of Mathematics, Miami University, Oxford, OH, 45056}

\date{}
\begin{document}
\maketitle
\begin{abstract} It was recently shown that each C*-algebra generated by a faithful irreducible representation of a finitely generated, torsion free nilpotent group is classified  by its ordered K-theory.  For the three step nilpotent group $UT(4,\Z)$ we calculate the ordered K-theory of each C*-algebra generated by a faithful irreducible representation of $UT(4,\Z)$ and see that they are all simple A$\T$ algebras.  We also point out that there are many simple \emph{non} A$\T$ algebras generated by irreducible representations of nilpotent groups.

\end{abstract}
\section{Introduction} 
The last few years witnessed several breakthroughs in the theory of simple, nuclear C*-algebras.  The hands of R\o rdam, H. Lin, Z. Niu, Winter, Matui and Sato  joined to show that if a C*-algebra satisfies several abstract properties (see Theorem \ref{thm:bigol}) then it necessarily has the concrete property of being an approximately subhomogenous algebra and is moreover classified\footnote{See Definition \ref{def:classifiable} for our working definition of ``classified"}  by its Elliott invariant. 

In other words, if a C*-algebra $A$ satisfies Theorem \ref{thm:bigol} then it is an inductive limit of subhomogeneous C*-algebras---but knowledge of the Elliott invariant of $A$ provides (in theory) an explicit decomposition of $A$ as a limit of subhomogenous C*-algebras and therefore a wealth of information about its structure.   This provides an entire new avenue of study for those important classes of C*-algebras that do not  have an obvious inductive limit structure, e.g. C*-algebras produced by dynamical systems or group representations.
We travel this new avenue of study  by calculating the Elliott invariant of C*-algebras generated by faithful irreducible representations of the three step nilpotent group $UT(4,\Z).$

 The odd-numbered authors showed that if $\Gamma$ is a finitely generated torsion free nilpotent group and $\pi$ is a faithful irreducible representation of $\Gamma$, then $C^*(\pi(\Gamma))$ satisfies Theorem \ref{thm:bigol}.  Therefore our present results combined with Theorem \ref{thm:shitinyourhat} determine the structure of C*-algebras generated by faithful irreducible representations of $UT(4,\Z).$

Our current program was carried out (without the aid of Theorem \ref{thm:bigol}) for many two-step nilpotent groups. Perhaps 
the best known example is Elliott and Evans's work on the irrational rotation algebras in \cite{Elliott93}.  They showed that the C*-algebras generated by faithful irreducible representations of the discrete Heisenberg group $\HH_3$ are A$\T$ algebras. The Elliott invariant had long been known for these algebras by the work of Rieffel and Pimsner and Voiculescu \cite{Rieffel81,Pimsner80,Pimsner80a}. 

It follows from the main theorem in  Phillips's preprint  \cite{Phillips06} that the C*-algebra generated by a faithful irreducible representation of a finitely generated two-step nilpotent group is an A$\T$ algebra. This is essentially the reason that we focus on the three step nilpotent group $UT(4,\Z)$ as it is the least complicated, most natural group covered by Theorem \ref{thm:shitinyourhat}, but not by Phillips's theorem. 

We briefly explain the technical aspects of our calculations. The C*-algebras generated by faithful irreducible representations of $UT(4,\Z)$ are parameterized by the irrational numbers in $(0,1).$ For $\alpha\in (0,1)$ and irrational let us denote by $B_\alpha$ the C*-algebra generated by this representation.  The Pimsner-Voiculescu six term exact sequence and a straightforward application of a theorem of Packer and Raeburn \cite{Packer90} show that $K_i(B_\alpha)\cong \Z^{10}$ for $i=0,1.$  The bulk of our work is then devoted to divining the order structure on $K_0(B_\alpha).$ We do this by locating a well-behaved (with respect to order structure) finite index subgroup $G\leq UT(4,\Z)$ and applying Pimsner's \cite{Pimsner85} to show that
the order on $K_0(B_\alpha)$ is determined by the representation restricted to $G.$  In particular we show the order on $K_0(B_\alpha)$ is given by the hyperplane with normal vector $(1,\alpha,\alpha^2,0,...,0).$

With the K-theory  of each $B_\alpha$ in hand, in Section \ref{sec:aeic}, we address when $B_\alpha\cong B_\beta.$ It is fairly easy to see that if $\alpha$ is transcendental or algebraic with minimal polynomial of degree greater than or equal to five, we have $B_\alpha\cong B_\beta$ if and only if $\alpha=\pm\beta \textup{ mod }\Z.$ On the other hand, if the degree of the minimal polynomial for $\alpha$ is less than or equal to four the situation is much more interesting (see Theorem \ref{thm:isomorphism_criterion} and following examples) and contrasts with the case of the irrational rotation algebras.  An extremely crude summation of the works \cite{Rieffel81,Pimsner80,Pimsner80a} is ``two irrational rotation algebras are isomorphic if and only if they are obviously isomorphic."
Theorem \ref{thm:isomorphism_criterion} shows cases with $B_\alpha\cong B_\beta$ that are not obviously isomorphic, i.e. the classification theorem (Theorem \ref{thm:bigol}) is essential.

As mentioned above, Phillips showed that every C*-algebra generated by an irreducible representation of a finitely generated two-step nilpotent group is an A$\T$ algebra.  All of the algebras considered here also turn out to be A$\T$ algebras.  For the sake of completeness we finish the paper with Section \ref{sec:blah}  by pointing out that there are many  C*-algebras generated by faithful irreducible representations of three-step nilpotent groups that are not A$\T$ algebras.

\section{Preliminaries} In this section we define our objects of study and recall the necessary C*-algebraic background.  For information on  the  properties  of A$\T$ algebras we refer the reader to R\o rdam's monograph \cite{Rordam02}.
The following major theorem is crucial to our investigations.
\begin{theorem}[R\o rdam, H. Lin, Z. Niu, Winter, Matui and Sato \textup{\cite{Lin04, Rordam04, Lin08, Winter12, Matui12,Matui14a}}] \label{thm:bigol} Let $A$ and $B$ be  unital, separable, simple, nuclear, quasidiagonal C*-algebras with unique tracial states and finite nuclear dimension that satisfy the universal coefficient theorem.  Then $A$ is an approximately subhomogeneous C*-algebra.  Moreover, if $(K_0(A),K_0(A)^+,[1_A], K_1(A))\cong (K_0(B),K_0(B)^+,[1_B], K_1(B))$, then $A\cong B.$ 
\end{theorem}

\begin{definition} \label{def:classifiable} In this paper, when we say a C*-algebra is \textbf{classifiable} we mean that it satisfies the hypotheses of Theorem \ref{thm:bigol}. In reality, the term ``classifiable" refers to a much larger class of C*-algebras.
We sacrificed generality for clarity in our definition.  See the recent preprint  \cite{Gong15} for a much more general  definition of classifiable.
\end{definition}
\begin{theorem}[See \textup{\cite{Eckhardt14,Eckhardt14b}}] \label{thm:shitinyourhat} Let $\Gamma$ be a torsion free finitely generated nilpotent group and $\pi$ a faithful irreducible representation of $\Gamma.$  Then $C^*(\pi(\Gamma))$ is classifiable (Definition \textup{\ref{def:classifiable}}).
\end{theorem}

\begin{definition} \label{def:fuckface}
For a group $\Gamma$, let $Z(\Gamma)\leq \Gamma$ denote the center of $\Gamma.$  For $a,b\in \Gamma,$ set $[a,b]=aba^{-1}b^{-1}.$
The \textbf{unitriangular subgroup} of $GL(4,\Z)$ is defined as

 \begin{equation} \label{eq:U4def}
UT(4,\Z)=\left\{  a=\left(  \begin{array}{llll}  1& a_{12} & a_{13}& a_{14}\\
                                                           & 1& a_{23}& a_{24}\\
                                                            && 1 & a_{34}\\
                                                            &&&1 \end{array} \right)  : a_{ij}\in \Z   \right\}. 
                                                            \end{equation}

\end{definition}
The following subgroup of $UT(4,\Z)$ plays a key role in our calculations,
\begin{equation*}
\quad \HH_4=\{ a\in UT(4,\Z):a_{23}=0 \}.
\end{equation*}
For each $1<i<j\leq 4$, let $e_{ij}\in M_4(\Z)$ be the $(i,j)$-matrix unit. One easily verifies the following commutation relations:
\begin{equation}
[1+e_{ij},1+e_{k\ell}]=1+\delta_{jk}e_{i\ell}-\delta_{i\ell}e_{kj}. \label{eq:commrel}
\end{equation}
Notice that 
\begin{equation} \label{eq:unicenter}
Z(UT(4,\Z))=\left\{  \left(  \begin{array}{llll}  1& 0 &0 & a\\
                                                           & 1&0 & 0\\
                                                            && 1 &0\\
                                                            &&&1 \end{array} \right)  : a\in \Z    \right\}\cong \Z.
\end{equation}
\subsection{Representation-theoretic description} 
\begin{definition} \label{def:tracedef} Let $\theta\in \R$.  Define the trace on $UT(4,\Z)$ as follows:
\begin{equation*}
\tau_\theta(x)=\left\{ \begin{array}{ll} e^{2\pi i x\theta} & \textrm{ if } x\in Z(UT(4,\Z))\\
                                                          0 & \textrm{ if } x\not\in Z(UT(4,\Z))\\
                                                          \end{array}
                                                        \right. .
\end{equation*}
Let $\pi_\theta$ denote the GNS representation of $UT(4,\Z)$ associated with $\tau_\theta.$
\end{definition}
Let $\pi$ be an irreducible, faithful, unitary representation of $UT(4,\Z).$  It is well known (see \cite{Moore76, Howe77, Carey84} or the introduction of \cite{Eckhardt14})
that there is an irrational $\theta\in \R$ such that  $C^*(\pi(UT(4,\Z)))\cong C^*(\pi_\theta(UT(4,\Z))).$
\begin{definition} For each $\theta\in \R$ we define
\begin{equation*}
B_\theta=C^*(\pi_\theta(UT(4,\Z))).
\end{equation*}
\end{definition}
\subsection{Crossed product construction} In order to describe the order structure on $K_0(B_\theta)$ we describe $B_\theta$ as a crossed product.
Recall the definition of $\HH_4$ in Definition \ref{def:fuckface}.
 \begin{lemma} \label{lem:Atprod} Let $\theta$ be irrational and $A_\theta$ denote the irrational rotation algebra associated with $\theta.$ Then
$C^*(\pi_\theta(\mathbb{H}_4))\cong A_\theta\otimes A_\theta$ ($\pi_\theta$ is the representation from Definition \ref{def:tracedef}).
\end{lemma}
\begin{proof} By the relations in (\ref{eq:commrel}), one sees that $C^*(\pi_\theta(1+e_{12}), \pi_\theta(1+e_{24}))$ and $C^*(\pi_\theta(1+e_{13}), \pi_\theta(1+e_{34}))$ are two commuting copies of $A_\theta.$  The conclusion then follows by the simplicity of $A_\theta$, the nuclearity of $A_\theta$ and  Takesaki's theorem on the simplicity of simple minimal tensor products (see \cite[Corollary IV.4.21]{Takesaki02}).  
\end{proof}

We now describe the conjugation action of $1+e_{23}$ on $A_\theta\otimes A_\theta.$ Define the automorphism $\beta:\mathbb{H}_4\rightarrow \mathbb{H}_4$ by 
\begin{equation}
\beta(x)=(1+e_{23})x(1+e_{23})^{-1}=[1+e_{23},x]x. \label{eq:betadef}
\end{equation}
This combined with the relations in (\ref{eq:commrel}) produces
\begin{enumerate}
\item[] $\beta(1+e_{12})=1+e_{12}-e_{13}$,
\item[] $\beta(1+e_{13})=1+e_{13}$,
\item[] $\beta(1+e_{34})=1+e_{24}+e_{34}$,
\item[] $\beta(1+e_{24})=1+e_{24}$.
\end{enumerate}
Summarizing the above discussion we obtain
\begin{theorem} \label{thm:desc} Let $u,v$ be standard generators of $A_\theta.$  Then $B_\theta\cong (A_\theta\otimes A_\theta) \rtimes_\beta \Z$
where 
\begin{enumerate}
\item[] $\beta(u\otimes 1)=u\otimes u^{-1}$,
\item[] $\beta(v\otimes 1)=v\otimes 1$,
\item[] $\beta(1\otimes u)=1\otimes u$,
\item[] $\beta(1\otimes v)=v\otimes v$.
\end{enumerate}
\end{theorem}

\subsection{Twisted group C*-algebra} In order to calculate the K-groups of $B_\theta$ we 
 describe  $B_\theta$ as a twisted
group C*-algebra. Everything  in this section will be well-known to experts in twisted group C*-algebras.  
We recall the necessary definitions (see for example Section 1 of \cite{Packer89}) for the non-experts.  Let $\Gamma$ be a discrete group and  $\sigma:\Gamma\times \Gamma\to \mathbb{T}$  a 2-cocycle (also referred to as a \emph{multiplier}). The involutive Banach algebra $\ell^1(\Gamma, \sigma)$ is formed with the following multiplication and involution
\begin{align}\label{eqn100}
f\ast g(s) :=\sum_{t\in \Gamma} f(t)g(t^{-1}s)\sigma(t,t^{-1}s),\quad f^*(s)=\overline{\sigma(s^{-1},s)f(s^{-1})}.
\end{align}
The left regular representation of $\ell^1(\Gamma,\sigma)$ is defined on $B(\ell^2(\Gamma))$ as
\begin{equation*}
\lambda(f)(g)(t)=\sum_{s\in \Gamma}\sigma(s,s^{-1}t)f(s)g(s^{-1}t).
\end{equation*}
The \emph{reduced twisted group C*-algebra} is defined as $C^*_r(\Gamma,\sigma)=C^*(\lambda(\ell^1(\Gamma,\sigma))).$ 
\\\\
Let $Z$ be the center of $UT(4,\Z)$ and  $C = \{x\in UT(4,\Z):x_{14}=0\}\subseteq UT(4,\Z).$  Notice that $C$ is a complete choice of coset representatives for $UT(4,\Z)/Z.$ Let $c:UT(4,\Z)/Z\rightarrow C$ be the unique lifting of the quotient map. Following \cite{Packer92} 
for each $\theta\in \R$,  we define the 2-cocycle $\omega_\theta: UT(4,\Z)/Z\times UT(4,\Z)/Z\rightarrow \T$ by 
\begin{align*}
\omega_{\theta}(xZ,yZ)=\tau_\theta(c(xZ)c(yZ)c(xyZ)^{-1}).
\end{align*}
\begin{proposition} \label{prop:twistgroup}
Let  $\omega_\theta$ be the cocycle from above.  Then $B_\theta$ is isomorphic to the reduced twisted group C*-algebra $C_r^*(UT(4,\Z)/Z,\omega_\theta).$
\end{proposition}
 \begin{proof}
 Let $L^2(UT(4,\Z),\tau_\theta)$ denote the Hilbert space associated to the GNS
representation of $\tau_{\theta}$.

It was shown in \cite[Lemma 2.4]{Eckhardt14}, that $\{\delta_x:x\in
C\}$ is an orthonormal basis for  $L^2(UT(4,\Z),\tau_\theta)$ and
$W:L^2(UT(4,\Z),\tau_\theta)\to \ell^2(UT(4,\Z)/Z)$, given by $W\delta_x=\delta_{xZ}$, is
unitary.  Moreover \cite[Lemma 2.4]{Eckhardt14} shows that $B_\theta$ is generated by $\{ \pi_\theta(x): x\in C  \}.$
\newline
Very easy  calculations and \cite[Lemma 2.4]{Eckhardt14} show that for each $x\in C$ we have
\begin{equation*}
W\pi_\theta(x)=\lambda(xZ)W.
\end{equation*}
It then follows that $B_\theta\cong C^*_r(UT(4,\Z)/Z,\omega_\theta).$
\end{proof}

\section{Computation of $K_*$} We now use the fact that each $B_\theta$ is a twisted group C*-algebra to calculate their unordered K-groups. Throughout this section we set
\begin{equation*}
\Gamma=UT(4,\Z)/Z(UT(4,\Z)).
\end{equation*}
\begin{lemma} \label{lem:groupK} We have 
\begin{equation*}
K_0(C^*(\Gamma))=K_1(C^*(\Gamma))=\Z^{10}.
\end{equation*}
\end{lemma}
\begin{proof} Note that $\Gamma\cong \Z^4\rtimes_\alpha \Z$ where $\la e_{12},e_{13},e_{24},e_{34}  \ra /Z(UT(4,\Z))\cong \Z^4$ and the automorphism $\alpha$ is implemented by conjugation of $e_{23}$ mod $Z(UT(4,\Z)).$ 

In other words, let $x_1,x_2,x_3,x_4$ be a free basis of $\Z^4$ and define
\begin{equation}
\alpha(x_1)=x_1-x_2, \quad \alpha(x_2)=x_2, \quad \alpha(x_3)=x_3, \quad \alpha(x_4)=x_3+x_4, \label{eq:shitforbrains}
\end{equation}
then $\Gamma\cong \Z^4\rtimes_\alpha \Z.$

It is well known (see for example \cite[2.1]{Elliott84} or \cite{Blackadar98}) that the graded ring $K_*(C(\T^4))=K_0(C(\T^4))\oplus K_1(C(\T^4))$ can be identified with the exterior algebra of $4$ generators $e_1,e_2,e_3,e_4$ over $\Z.$ Furthermore, $K_0$ is identified with those terms of even degree and the standard coordinate functions in $C(\T^4)$ correspond to $e_1,e_2,e_3,e_4$. 
The induced action of $\alpha$ on $K_*(C(\T^4))$ is a ring automorphism.   This fact combined with (\ref{eq:shitforbrains}) shows that
$\alpha_*(x)=x$ for 
\begin{equation*}
x\in \{ 1, e_2,e_3, e_1\wedge e_2, e_2\wedge e_3, e_3\wedge e_4, e_1\wedge e_2\wedge e_3, e_2\wedge e_3\wedge e_4, e_1\wedge e_2\wedge e_3\wedge e_4 \}.
\end{equation*}
 Furthermore,
\begin{equation*}
\begin{array}{ll}
\alpha_*(e_1)= e_1-e_2,&\alpha_*(e_1\wedge e_3\wedge e_4)=e_1\wedge e_3\wedge e_4-e_2\wedge e_3\wedge e_4, \\
\alpha_*(e_4)=e_3+e_4,& \alpha_*(e_1\wedge e_2\wedge e_4)=e_1\wedge e_2\wedge e_3+e_1\wedge e_2\wedge e_4,
\end{array}
\end{equation*}
which determines the homomorphism $\alpha_*:K_1(C(\T^4))\rightarrow K_1(C(\T^4)).$  It follows that 
\begin{equation}
K_1(C(\T^4))/(\textup{id}-\alpha_*)(K_1(C(\T^4)))\cong \Z^4. \label{eq:K1quotient}
\end{equation}
Similarly, the following calculations determine the $K_0$ counterpart:
\begin{equation*}
\begin{array}{l}
\alpha_*(e_1\wedge e_3)=e_1\wedge e_3-e_2\wedge e_3,\quad \alpha_*(e_2\wedge e_4)=e_2\wedge e_3+e_2\wedge e_4,\\
\alpha_*(e_1\wedge e_4)=e_1\wedge e_3+e_1\wedge e_4-e_2\wedge e_3-e_2\wedge e_4.\\
\end{array}
\end{equation*}
It follows that 
\begin{equation}
K_0(C(\T^4))/(\textup{id}-\alpha_*)(K_0(C(\T^4)))\cong \Z^2. \label{eq:K0quotient}
\end{equation}
As was implied above, we have $K_0(C(\T^4))\cong K_1(C(\T^4))\cong \Z^8.$  This fact combines with (\ref{eq:K1quotient}), (\ref{eq:K0quotient}) and the Pimsner-Voiculescu six term exact sequence for crossed products \cite[Theorem 2.4]{Pimsner80} to prove the claim.

\end{proof}
\begin{corollary} \label{cor:KZ10} Let $\theta\in \R.$   Then 
\begin{equation*}
K_i(B_\theta)\cong \Z^{10}\quad \textrm{ for }i=0,1.
\end{equation*}
\end{corollary}
\begin{proof}
By Proposition \ref{prop:twistgroup}, we have $B_\theta\cong
C^*(\Gamma,\omega_{\theta})$. When $\theta=0$,
$B_\theta$ is isomorphic to the group C*-algebra
$C^*(\Gamma)$.
For any two $\theta_1,\theta_2$, the
cocycles $\omega_{\theta_1}$ and $\omega_{\theta_2}$ are homotopic. Indeed, define
$r(t):=t\theta_1 + (1-t)\theta_2$, and a homotopy of 2-cocycles
$\tilde{\omega}:\Gamma\times \Gamma\times [0,1]\to \mathbb{T}$ by
$\tilde{\omega}(\cdot,\cdot,t):=\omega_{r(t)}(\cdot,\cdot)$.  

Since $\Gamma$ is discrete in the simply connected (contractible actually) nilpotent Lie group $UT(4,\R)/Z(UT(4,\R))$, it follows from
Theorem 4.2 of \cite{Packer90} that $K_*(B_\theta)\cong K_*(C^*(\Gamma)).$
\end{proof}

\section{Range of the trace} Let $A$ be a unital C*-algebra and $\tau$ a tracial state on $A.$ We also denote by $\tau$ the state on $K_0(A).$ In general it is difficult to calculate the range $\tau(K_0(A))\subset \R.$

Pimsner showed in \cite{Pimsner85} that in the case of crossed products by free groups an examination of the de la Harpe-Skandalis determinant \cite{Harpe84} can sometimes lead to a satisfying  description of the range of the trace of the crossed product.  Pimsner's ideas work particularly well for the cases at hand. 

We very briefly recall the necessary background from \cite{Harpe84} and \cite{Pimsner85} (see also \cite{Blackadar98}) and refer the reader to \cite{Pimsner85} for more information and proofs of the claims made below.

Let $U(A)$ denote the unitary group of a C*-algebra.  Let $U_\infty(A)$ be the inductive limit of $U(M_n(A))$ in the usual way.  Let $U_\infty(A)_0$ denote the connected component of the identity in $U_\infty(A).$
  
For a piecewise differentiable path $\xi:[0,1]\rightarrow U_\infty$ one defines
\begin{equation*}
\Delta_\tau(\xi)=\frac{1}{2\pi i}\int_0^1 \tau(\xi'(t)\xi(t)^{-1})\,dt.
\end{equation*}
The map $\Delta_\tau$ is constant on homotopy classes with fixed endpoints. 
Let $\xi$ be a differentiable path of unitaries and $n\geq1$ an integer.
From the following easily derivable formula,
\begin{equation*}
\frac{d}{dt}\xi^n(t)=n\xi'(t)\xi^{n-1}(t)-\sum_{i=1}^{n-1}[\xi'(t),\xi^{n-i}(t)]\xi(t)^{i-1},
\end{equation*}
one deduces that
\begin{equation} \label{eq:powerrule}
\Delta_\tau(\xi^n)=n\Delta_\tau(\xi). 
\end{equation}
\begin{lemma} \label{lem:trivialcommutator} Let $u,v\in U_n(A)$ be unitaries.  Then for any piecewise differentiable path  $\xi$ from $1$ to $uvu^{-1}v^{-1}$ we have $\Delta_\tau(\xi)=0.$
\end{lemma}
\begin{proof} Let $w(t)=\left(\begin{array}{ll} \cos(\pi t/2) & \sin(\pi t/2)\\  \sin(\pi t/2) & -\cos(\pi t/2)  \end{array}\right)$ and define
\begin{equation*}
\xi(t)=\left(\begin{array}{ll} u & \\   & 1 \end{array}\right)w(t)\left(\begin{array}{ll} v & \\   & 1 \end{array}\right)w(t)\left(\begin{array}{ll} u^{-1} & \\   & 1 \end{array}\right)w(t)\left(\begin{array}{ll} v^{-1} & \\   & 1 \end{array}\right)w(t).
\end{equation*}
It is then a straightforward calculation to see that $\Delta_\tau(\xi)=0.$
\end{proof}

Under the Bott isomorphism $K_0(A)\cong K_1(SA)$, one sees that $\Delta_\tau$ restricted to $K_0(A)$ is just $\tau.$ Let $q:\R\rightarrow \R/\tau(K_0(A))$ be the quotient map.  Then for any unitary $u\in U_\infty(A)_0$ and any piecewise differentiable path $\xi$ from 1 to $u$ the map
\begin{equation*}
\underline{\Delta}_\tau(u)=q\circ \Delta_\tau(\xi)
\end{equation*}
is a well-defined group homomorphism. 

 Suppose now that $\alpha$ is an automorphism of $A$ and that $\tau=\tau\circ\alpha.$  Define the group homomorphism $\underline{\Delta}_\tau^\alpha:\textup{ker}(K_1(\textup{id})-K_1(\alpha))\rightarrow \R/\tau(K_0(A))$ by
\begin{equation*}
\underline{\Delta}_\tau^\alpha([u]_1)=\underline{\Delta}_\tau(u\alpha^{-1}(u^{-1})).
\end{equation*}
Finally, we have
\begin{theorem}[Pimsner \cite{Pimsner85}] \label{thm:Pimsnertracerange}The following sequence is exact:
\begin{equation*}
0\longrightarrow \tau(K_0(A))\longrightarrow \tau(K_0(A\rtimes_\alpha\Z))\longrightarrow \underline{\Delta}_\tau^\alpha(\textup{ker}(K_1(\textup{id})-K_1(\alpha)))\longrightarrow0,
\end{equation*}
where the first map is the inclusion and the second is the quotient $q:\R\rightarrow \R/\tau(K_0(A)).$
\end{theorem}

It is well known that $A_\theta$ arises as a crossed product $C(\T)\rtimes \Z$, hence we may apply the K\"unneth formula \cite{Schochet82}, which  provides
\begin{equation*}
K_0(\Aa)\cong \Big(K_0(A_\theta)\otimes K_0(A_\theta)\Big)\oplus\Big( K_1(A_\theta)\otimes K_1(A_\theta)     \Big). 
\end{equation*}
Let $\tau$ be the unique trace on $A_\theta\otimes A_\theta.$  We also denote by $\tau$ the extension of this trace to $B_\theta.$ Pimsner and Voiculescu showed in \cite{Pimsner80} that the range of the trace on $K_0(A_\theta)$ is $\Z+\theta\Z.$ From this it follows that 
\begin{equation*}
\Z+\theta\Z+\theta^2\Z\subseteq \tau(K_0(A_\theta\otimes A_\theta)).
\end{equation*}
Moreover from the description of the map in the K\"unneth formula one checks that 
\begin{equation*}
\textup{ker}(\tau)\supseteq K_1(A_\theta)\otimes K_1(A_\theta), 
\end{equation*}
hence
\begin{equation}
\tau(K_0(\Aa))= \Z+\theta\Z+\theta^2\Z. \label{eq:tptracerange}
\end{equation}
\subsection{A finite index subgroup} Eventually we will apply Pimsner's ideas of the previous section to compute the range of the trace for $B_\theta\cong A_\theta\otimes A_\theta\rtimes_\beta\Z.$ As one might expect from Pimsner's Theorem \ref{thm:Pimsnertracerange}, a good description of $\textup{ker}(K_1(\textup{id})-K_1(\beta))$ is helpful.  In our case a complication arises because the automorphism $\beta$ of $A_\theta\otimes A_\theta$ mixes up the tensor factors making it difficult to describe $\textup{ker}(K_1(\textup{id})-K_1(\beta)).$  In this section we look at a subalgebra $A\subseteq B_\theta$ that is isomorphic to a crossed product of $A_{2\theta}\otimes A_{2\theta}$ by an automorphism that does factor as a tensor product of automorphisms of $A_{2\theta}$ (providing an easy path to  the range of the trace calculation for $A$).  The algebra $A$ is ``big enough" to then yield the range of trace calculation for $B_\theta.$

Recall the Heisenberg group $\HH_4$ defined in (\ref{eq:U4def}), and the automorphism $\beta$ from (\ref{eq:betadef}). Then $\HH_4/Z(\HH_4)\cong \Z^4.$  Since $\beta$ fixes the center, it drops to an automorphism of $\Z^4$ (which we still denote by $\beta$).   The matrix for $\beta\curvearrowright\Z^4 $  with respect to the basis $\{  1+e_{12},1+e_{13},1+e_{24},1+e_{34}     \} \textup{ mod } Z(\HH_4)$ is
\begin{equation*}
\beta=\left(  \begin{array}{rccc}  1&0 &&\\
                                                      -1&1&&\\
                                                      &&1&1\\
                                                       && 0&1   \end{array}  \right).
\end{equation*}
Then the following two subgroups of $\Z^4$ are invariant under $\beta$ and $\beta^{-1}:$
\begin{equation} \label{eq:X+Y}
\widetilde{X}=\left\la   \left( \begin{array}{r}  0\\ -1\\ 1\\ 0  \end{array}  \right),  \left( \begin{array}{r}  1\\ 0\\ 0\\ 1  \end{array}  \right)  \right\ra, \quad 
\widetilde{Y}=\left\la   \left( \begin{array}{r}  0\\ -1\\ -1\\ 0  \end{array}  \right),  \left( \begin{array}{r}  1\\ 0\\ 0\\ -1  \end{array}  \right)  \right\ra.
\end{equation}
Let $e_1,e_2,e_3,e_4$ be the standard basis for $\Z^4$ and $\pi:\Z^4\rightarrow \Z^4/\la\widetilde{X},\widetilde{Y}\ra$ the quotient map.  Then $3\pi(e_1)=\pi(e_1)=\pi(e_4)\neq \pi(e_2)=\pi(e_3)=3\pi(e_3)$, so
\begin{equation}\label{eq:findex}
\Z^4/\la\widetilde{X},\widetilde{Y}\ra\cong \Z/2\Z\times \Z/2\Z.
\end{equation}
Now set $X$ (resp. $Y$) equal to the inverse image of $\widetilde{X}$ (resp. $\widetilde{Y}$) under the quotient map $\HH_4\rightarrow \HH_4/Z(\HH_4).$
\begin{definition} \label{def:uivi}
 Set $u_1=\pi_\theta(1-e_{13}+e_{24})$, $v_1=\pi_\theta(1+e_{12}+e_{34})$, $u_2=\pi_\theta(1-e_{13}-e_{24})$, $v_2=\pi_\theta(1+e_{12}-e_{34}).$
Then $C^*(u_1,v_1)=C^*(\pi_\theta(X))\cong A_{2\theta}\cong C^*(\pi_\theta(Y))=C^*(u_2,v_2)$ and  by similar reasoning as in Lemma \ref {lem:Atprod} we have $C^*(\{ u_i,v_i: i=1,2\})\cong C^*(u_1,v_1)\otimes C^*(u_2,v_2)\cong A_{2\theta}\otimes A_{2\theta}.$ 
\end{definition}
Furthermore we have we have $\beta|_{A_{2\theta}\otimes A_{2\theta}}=\beta_1\otimes \beta_2$ where
\begin{equation}
\beta_j(u_j)=u_j, \quad \textrm{ and }\quad \beta_j(v_j)=\exp((-1)^j 2\pi i\theta)u_jv_j, \textrm{ for }j=1,2.
\end{equation}
The above information combines with \cite[Corollary 2.5]{Pimsner80a} to provide
\begin{lemma} \label{lem:KunnethPV} For $j=1,2$ we have $K_0(\beta_j)=id$ and $K_1(\beta_j)([u_j]_1)=[u_j]_1$ and $K_1(\beta_j)([v_j]_1)=[u_jv_j]_1.$
\end{lemma}

Let $G\leq \HH_4$ be the subgroup generated by $X$ and $Y$. Then $|\HH_4/G| =4$ by (\ref{eq:findex}).  Let $e=x_0,x_1,x_2,x_3\in \HH_4$ be $\HH_4/G$ coset representatives. As in Definition \ref{def:uivi}, we have  $A_{2\theta}\otimes A_{2\theta}\cong C^*(\pi_\theta(G)).$

Consider the unitary representation $\pi_\theta|_G$ of $G.$ The unitary representation $\pi_\theta$ of $\HH_4$ is just the induced representation $\textup{Ind}_G^{\HH_4}(\pi_\theta|_{G}).$  It then follows from  the general theory of induced representations (see for example, \cite[Appendix E]{Bekka08}) and C*-algebras that there is an  embedding $\sigma:A_\theta\otimes A_\theta\rightarrow M_{|\HH_4/G|}(A_{2\theta}\otimes A_{2\theta})=M_4(A_{2\theta}\otimes A_{2\theta})$ such that 
\begin{equation*}
\sigma(\pi_\theta(t))=\left(\begin{array}{cccc} \pi_\theta(t) &&&\\ & \pi_\theta(x_1^{-1}tx_1) &&\\ && \pi_\theta(x_2^{-1}tx_2)&\\ &&&\pi_\theta(x_3^{-1}tx_3)  \end{array} \right)
\end{equation*}
for all $t\in G.$ 
Notice that  each generator $u_1,u_2,v_1,v_2\in C^*(\pi_\theta(G))$ is mapped to a scalar multiple of itself under the automorphisms
Ad($\pi_\theta(x_i)$ for $i=1,2,3.$ In particular each of the automorphisms Ad($\pi_\theta(x_i))$ is homotopic to the identity. We have shown the following
\begin{proposition} \label{prop:index4} Let $\iota: A_{2\theta}\otimes A_{2\theta}\rightarrow A_{\theta}\otimes A_{\theta}$ be the inclusion map given by $C^*(\pi_\theta(G))\subseteq C^*(\pi_\theta(\HH_4)).$ Let $\sigma$ be as above.
Then $K_*(\sigma\circ \iota)=4\cdot id_{K_*(A_{2\theta}\otimes A_{2\theta})}.$  By the functoriality of $K_*$, and the fact that all the K-groups of $A_\theta\otimes A_\theta$ and $A_{2\theta}\otimes A_{2\theta}$ are  torsion free, we have  $K_*(\iota)$ and $K_*(\sigma)$ are both injective.
\end{proposition}
\subsection{Range of trace} 
Let $A$ and $B$ be C*-algebras. In \cite{Schochet82}, Schochet describes a $\Z/2\Z$ graded pairing
\begin{equation}
 \alpha: K_p(A)\otimes K_q(B)\rightarrow K_{p+q}(A\otimes B), \textrm{ }p,q\in\Z/2\Z \label{eq:Kunneth}
\end{equation}
that is an isomorphism when $A$ is in the ``bootstrap class" and $B$ has torsion free K-theory \cite[Theorem 2.14]{Schochet82}. 
Suppose now that $A$ and $B$ satisfy the conditions of the previous sentence.  

Let $\beta=\beta_1\otimes \beta_2$ be an automorphism of $A\otimes B.$ It is clear that $K_*(\beta_1)\otimes K_*(\beta_2)$ induces a graded automorphism of $K_*(A)\otimes K_*(B).$
Moreover, by the description of $\alpha$, a straightforward verification reveals that the following diagram commutes:
\begin{equation} \label{eq:autocomm}
\xymatrix{
 K_*(A)\otimes K_*(B) \ar[r]^{\alpha} \ar[d]_{K_*(\beta_1)\otimes K_*(\beta_2)} & K_*(A\otimes B)  \ar[d]^{K_*(\beta_1\otimes\beta_2)} \\
 K_*(A)\otimes K_*(B)  \ar[r]^{\alpha}  & K_*(A\otimes B) }.
\end{equation}
Since $A_\theta$ is isomorphic to a crossed product of the form $C(\T)\rtimes \Z$, it is in the bootstrap class by \cite[Proposition 2.7]{Schochet82}. Since $A_\theta$ has torsion free K-theory, we have an isomorphism in (\ref{eq:Kunneth}) when $A=B=A_\theta.$

Recall from \cite{Pimsner80a} that $K_0(A_{2\theta})\cong K_1(A_{2\theta})\cong \Z^2.$  Moreover for unitary generators $w_1$ and $w_2$, we have $\{ [w_1]_1,[w_2]_1  \}$ a free basis for $K_1(A_{2\theta}).$  By \cite{Schochet82} we have

\begin{equation*}
K_1(A_{2\theta}\otimes A_{2\theta})\cong [K_0(A_{2\theta})\otimes K_1(A_{2\theta})]\oplus [K_1(A_{2\theta})\otimes K_0(A_{2\theta})]\cong \Z^4\oplus \Z^4\cong \Z^8.
\end{equation*} 
Under this identification and the  ordered free bases $\{ [u_j]_1,[v_j]_1  \}$ of $K_1(A_{2\theta})$, by Lemma \ref{lem:KunnethPV}, we have
\begin{equation*}
K_1(\beta_1\otimes\beta_2)=\Big[  \textup{id}_{\Z^2}\otimes \left( \begin{array}{cc} 1 & 1\\ 0&1\\     \end{array}\right)   \Big]\oplus\Big[  \left( \begin{array}{cc} 1 & 1\\ 0&1\\     \end{array}\right)\otimes \textup{id}_{\Z^2} \Big].
\end{equation*}
We therefore have the following
\begin{lemma} \label{lem:kergen} The subgroup $\textup{ker}(K_1(\textup{id})-K_1(\beta_1\otimes \beta_2))\leq K_1(A_{2\theta\otimes 2\theta})$ is generated by the set
\begin{equation} \label{eq:kergen}
\{ [p\otimes u_2+(1-p)\otimes 1]_1, [u_1\otimes p+1\otimes (1-p)]_1: p\in A_{2\theta} \textup{ is a projection}  \}.
\end{equation}
\end{lemma}

\begin{theorem}\label{lem:tracerange}  We have
\begin{equation*}
\tau(K_0(B_\theta))=\tau(K_0((A_{\theta}\otimes A_{\theta})\rtimes_{\beta}\Z))=\Z+\theta\Z+\theta^2\Z.
\end{equation*}
\end{theorem}
\begin{proof}
By (\ref{eq:tptracerange}) and Theorem \ref{thm:Pimsnertracerange} it suffices to show that for all  $x\in \textup{ker}(K_1(\textup{id})-K_1(\beta))$ we have
$\underline{\Delta}_\tau^\beta(x)=0.$  By Proposition \ref{prop:index4}, for any $x\in \textup{ker}(K_1(\textup{id}_{A_\theta\otimes A_\theta})-K_1(\beta))$ we have $4x\in \textup{ker}(K_1(\textup{id}_{{A_{2\theta}}\otimes A_{2\theta}})-K_1(\beta)).$ Therefore by (\ref{eq:powerrule}) we need to show that for any $[w]_1\in \textup{ker}(K_1(\textup{id}_{{A_{2\theta}}\otimes A_{2\theta}})-K_1(\beta))$ and  differentiable path  $\xi$ from the identity to $w\beta^{-1}(w)$ in $(U_\infty)_0$ we have $\Delta_\tau(\xi)=0.$

By Lemma \ref{lem:kergen}  we only need to show this for those $w$ of the form in (\ref{eq:kergen}).  We will only show it for $w$ of the form 
$[p\otimes u_2+(1-p)\otimes 1]_1$ (one proves it for $[u_1\otimes p+1\otimes (1-p)]_1$ in exactly the same manner).

By a result of Rieffel \cite[Corollary 2.5]{Rieffel83} the projections $p$ and $\beta_1^{-1}(p)$ are unitarily equivalent.  Since $\beta_2(u_2)=u_2$ it follows that $p\otimes u_2+(1-p)\otimes 1$ is unitarily equivalent to $\beta(p\otimes u_2+(1-p)\otimes 1).$ The conclusion now follows from Lemma \ref{lem:trivialcommutator}.
\end{proof}
\section{Elliott invariants} We now gather the information of the proceeding sections to describe the Elliott invariants of the algebras $B_\theta.$ 
\begin{theorem} Let $\theta$ be irrational.  Then $B_\theta$ is a simple A$\T$ algebra with unique trace and $K_i(B_\theta)\cong \Z^{10}$ for $i=0,1.$ 
Let $\Theta=(1,\theta,\theta^2,0,...,0)\in \R^{10}.$  Then 
\begin{equation} \label{eq:poscone}
K_0(B_\theta)^+ =\{ 0 \}\cup \{x\in \Z^{10}: \la x,\Theta  \ra>0   \}.
\end{equation}
\end{theorem}
\begin{proof} The $K$-groups were calculated in Corollary \ref{cor:KZ10}.
Since $B_\theta$ is the C*-algebra generated by an irreducible representation of a finitely generated nilpotent group, it is simple with a unique trace (this is well known, see e.g. the introduction of \cite{Eckhardt14}).

It follows from \cite[Theorems 2.9 \& 4.4 ]{Eckhardt14b} that $B_\theta$ has strict comparison.  Since $B_\theta$ has unique trace, this means the order structure on $K_0(B_\theta)=\Z^{10}$ is completely determined by the range of the trace, which is  $\Z+\theta\Z+\theta^2\Z$ by Theorem \ref{lem:tracerange}.  This shows (\ref{eq:poscone}).

This description of $(K_0(B_\theta), K_0(B_\theta)^+)$ shows it  is a Riesz group, and therefore a dimension group by the Effros-Handelman-Shen theorem \cite{Effros80a}.
Therefore $B_\theta$ has the same Elliott invariant of an A$\T$ algebra by \cite{Elliott93a}.

Since $A_\theta$ is an A$\T$ algebra by \cite{Elliott93}, it follows that $A_\theta\otimes A_\theta$ is an A$\T^2$ algebra (it is actually an A$\T$ algebra see \cite[Proposition 3.25]{Rordam02}) and therefore satisfies the universal coefficient theorem by \cite{Rosenberg87}.  Also by \cite{Rosenberg87}, it follows that  $(A_\theta\otimes A_\theta)\rtimes_\beta\Z$ satisfies the universal coefficient theorem.
By \cite{Eckhardt14}, $B_\theta$ is quasidiagonal and by \cite{Eckhardt14b} $B_\theta$ has finite nuclear dimension.
By Theorem \ref{thm:bigol}, $B_\theta$ is therefore isomorphic to an A$\T$ algebra.

\end{proof}

\section{An  isomorphism criterion} \label{sec:aeic}
As mentioned in the introduction we intend to show there are irrational numbers $\theta,\eta$ such that $B_\theta\cong B_\eta$ but $B_\theta$ and $B_\eta$ are not ``obviously" isomorphic (i.e. $\theta\neq -\eta \textup{ mod }\Z$). This shows that one must use the classification theorem (Theorem \ref{thm:bigol}) to classify the algebras $\{ B_\theta: \theta\in (0,1)\setminus\Q \}$ amongst themselves. 

  By the results of the preceding sections it follows that $B_\theta$ and $B_\eta$ are isomorphic if and only if the
  ordered groups $(\Z^3,(1,0,0),P_\theta)$ and $(\Z^3,(1,0,0),P_\eta)$ with distinguished order unit are isomorphic, where $P_\theta =
  \{x\in\Z^3\;\mid\; x_1 + x_2\theta + x_3\theta^2 > 0 \}\cup\{ 0 \}$.

  An automorphism of $\Z^3$ to $\Z^3$ is implemented by some $A^t\in\GL(3,\Z)$ ( $t$ denotes transpose), and it is easy to see that $A^t$ sends
  $P_\theta$ to $P_\eta$ if and only if
  \begin{equation}
    \label{eq:theta-eta-reln} A\begin{pmatrix} 1 \\ \theta \\ \theta^2 \end{pmatrix} = \begin{pmatrix} 1 \\ \eta \\ \eta^2 \end{pmatrix}.
  \end{equation}
  We will find below that for most pairs $(\theta,\eta)$,~\eqref{eq:theta-eta-reln} holds if and only if $\eta = \pm \theta \pmod{\Z}$.
  This reflects the situation with the irrational rotation algebras $A_\theta$.  However, for certain
  $\theta$, there are more possibilities.

  Fix an irrational $\theta$.  Our goal is to more easily describe the relation between $\theta$ and $\eta$ given
  by~\eqref{eq:theta-eta-reln}.   Equation  (\ref{eq:theta-eta-reln}) defines an equivalence relation $\theta\sim\eta$ which extends
  the equivalence relation $\theta = \pm\eta \pmod{\Z}$, since if $\eta = \pm\theta + k$ with $k\in\Z$, then
  \[
    \begin{pmatrix}
      1 & 0 & 0 \\
      k & \pm 1 & 0 \\
      k^2 & \pm 2k & 1
    \end{pmatrix}
    \begin{pmatrix} 1 \\ \theta \\ \theta^2 \end{pmatrix}
     =
    \begin{pmatrix} 1 \\ \eta \\ \eta^2 \end{pmatrix}.
  \]
  We will describe the relation $\sim$ in several cases based on $\textup{deg}\theta$ (that is, the degree of
  the minimal polynomial of $\theta$ over $\Q$).


  \subsection{$\deg{\theta} > 2$} Throughout this subsection fix a $\theta$ with degree strictly bigger than two.  Suppose that $\theta\sim \eta.$ The degree restriction implies that the first row of $A$ must be $(1\ 0\
  0)$.  Also, since we are only counting the number of $\eta$'s which are distinct mod $\Z$ and modulo the automorphism
  $\eta\mapsto -\eta$, we may assume that $a_{21} = 0$ and $\det{A} = 1$.  Then we may write $A$ in the form
  \begin{equation}
    \label{eq:form-of-A}
    A = \begin{pmatrix}
      1 & 0 & 0 \\
      0 & a & b \\
      c & d & e
    \end{pmatrix}.
  \end{equation}
  By~\eqref{eq:theta-eta-reln},
  \[
    (a\theta + b\theta^2)^2 = \eta^2 = c + d\theta + e\theta^2,
  \]
 hence 
  \begin{equation}
    \label{eq:quadratic} b^2\theta^4 + 2ab\theta^3 + (a^2 - e)\theta^2 - d\theta - c = 0.
  \end{equation}
  We now split into three subcases.

  Case (a): $\theta$ has degree $> 4$.  Then the coefficients of the polynomial in~\eqref{eq:quadratic} must be zero.
  In particular, $b = 0$ and $e = a^2$.  Then $A$ is lower triangular, and since $\det{A} = 1$, it follows that $a = e =
  1$.  But then $\eta = \theta$.

  Case (b): $\theta$ has degree $4$.  Let $p(x) = x^4 + \lambda_3 x^3 + \lambda_2 x^2 + \lambda_1 x + \lambda_0$, with
  $\lambda_i\in\Q$, be the minimal polynomial for $\theta$.  Then it follows that
  \begin{align*}
    2ab & = b^2 \lambda_3, \\
    a^2 - e & = b^2 \lambda_2, \\
    - d & = b^2 \lambda_1, \\
    - c & = b^2 \lambda_0.
  \end{align*}
  Moreover, $ae - bd = \det(A) = 1$.  Note that if $b = 0$, then just as in case (a), we have $a = 1$ and hence
  $\eta = \theta$.  Otherwise, we have
  \begin{align*}
    a & = \frac{1}{2}b\lambda_3, \\
    e & = \frac{1}{4}b^2\lambda_3^2 - b^2\lambda_2, \\
    d & = - b^2 \lambda_1, \\
    c & = - b^2 \lambda_0,
  \end{align*}
  and hence $\mu b^3 = 1$, where
  \[
    \mu = \frac{1}{8}\lambda_3^3 - \frac{1}{2}\lambda_3\lambda_2 + \lambda_1.
  \]

  Case (c): $\theta$ has degree $3$.
  Let $p(x) = x^3 + \lambda_2 x^2 + \lambda_1 x + \lambda_0$ be the minimal polynomial for $\theta$.  Then, we have
  \[
    (2ab - b^2\lambda_2) \theta^3 + (a^2 - e - b^2\lambda_1)\theta^2 - (d + b^2\lambda_0)\theta - c = 0
  \]
  in which case,
  \begin{align*}
    a^2 - e - b^2\lambda_1 & = (2ab - b^2\lambda_2)\lambda_2, \\
    - (d + b^2\lambda_0) & = (2ab - b^2\lambda_2)\lambda_1, \\
    - c & = (2ab - b^2\lambda_2)\lambda_0.
  \end{align*}
  The above equations allow us to express $c$, $d$ and $e$ in terms of $a$ and $b$.  Then, again using $\det{A} = ae -
  bd = 1$, we have
  \begin{equation}
    \label{eq:cubic}  a^3 - 2\lambda_2 a^2 b + (\lambda_2^2 + \lambda_1)ab^2 + (\lambda_0 - \lambda_1 \lambda_2)b^3 =
    \det{A} = 1.
  \end{equation}
  The above equation is of the form $F(a,b) = 0$, where $F(x,y) = x^3 + p x^2y + q xy^2 + r y^3 - 1$ is a cubic
  polynomial with rational coefficients.  Its projectivization is $F(X,Y,Z) = X^3 + p X^2 Y + q X Y^2 + r Y^3 - Z^3$.
  \begin{lemma}
    The projective curve defined by $F(X,Y,Z) = 0$ is nonsingular, that is, there are no points $[X:Y:Z]\in\mathbb{P}^2$
    for which $F,\frac{\partial F}{\partial X}, \frac{\partial F}{\partial Y}$, and $\frac{\partial F}{\partial Z}$ all
    vanish simultaneously.
  \end{lemma}
  \begin{proof}
    Since $\frac{\partial F}{\partial Z} = 3Z^2$, any nonsingular point must lie in $\mathbb{R}^2$.  So suppose
    $(x,y)\in\R^2$ satisfy $\frac{\partial F}{\partial x}(x,y) = \frac{\partial F}{\partial y}(x,y) = 0$.
    Writing $t = x + \frac{1}{3} p y$, we have
    \[
      \frac{\partial F}{\partial x} = 3t^2 + \left(q - \frac{1}{3} p^2\right) y^2 = 0
    \]
    which implies $t^2 = \left(\frac{1}{9} p^2 - \frac{1}{3}q\right) y^2$.  Let $k = \frac{1}{9} p^2 - \frac{1}{3}q$.  Rewriting
    $\frac{\partial F}{\partial y}$ in terms of $t$, and using that $t^2 = k y^2$, we get
    \[
      \frac{\partial F}{\partial y} = \left(2pk - 6k^{3/2} + 3r - \frac{1}{3}pq\right) y^2 = 0.
    \]
    If $y = 0$, then $t = 0$ and hence $x = 0$.  But then $F(x,y) = -1$.  Hence we may assume instead that
    \[
      \alpha = 2pk - 6k^{3/2} + 3r - \frac{1}{3}pq = 0.
    \]
    Now, rewriting $F$ in terms of $t$, and again using that $t^2 = ky^2$, we get
    \[
      F(x,y) = \left(-2k^{3/2} - \frac{1}{27}p^3 + pk + r\right)y^3 - 1.
    \]
    Writing $\beta$ for the coefficient of $y^3$ above, we see that
    \[
      3\beta = -6k^{3/2} - \frac{1}{9} p^3 + 3pk + 3r = \alpha + pk - \frac{1}{9} p^3 + \frac{1}{3}pq = \alpha.
    \]
    Since $\alpha = 0$, it follows that $F(x,y) = -1$, and hence $(x,y)$ does not lie on the curve.
  \end{proof}
  Siegel's theorem (\cite{Silverman-Tate}) implies that the curve
  $F(x,y) = 0$ can only have finitely many integral
  points.  It follows that there are only finitely many choices for $A$, and hence for $\eta$.
  
%

 \subsection{$\deg\theta=2$} Here we take a different strategy for determining how many $\eta$ may satisfy
  $\theta \sim \eta$, by studying the groups $G_\eta = \Z + \eta\Z + \eta^2\Z$.  Note that if $\theta \sim \eta$, then
  $G_\theta = G_\eta$.  So suppose this is true for some $\eta$.  Since $\theta$ has degree $2$, we may write
  \[
    \theta = \frac{a + b\sqrt{k}}{c}
  \]
  where $a,b,c\in\Z$ have no common factors, $k\in\Q$, and $k$ can be written $r/s$ where $r,s\in\Z$ have no common
  factors and are each square free.  Since
  \[
    \Q + \sqrt{k} \Q = \Q + \theta \Q + \theta^2 \Q = \Q + \eta\Q + \eta^2\Q,
  \]
  it follows that $\deg{\eta} = 2$ as well, and moreover
  \[
    \eta = \frac{x + y\sqrt{k}}{z}
  \]
  for some $x,y,z\in\Z$ with no common factors.  Now, we have
  \[
    G_\theta = \Z + \left(\frac{a + b\sqrt{k}}{c}\right)\Z + \left(\frac{a^2 + b^2 k + 2ab \sqrt{k}}{c^2}\right)\Z
    \subseteq \frac{1}{c^2 s}\left(\Z + \sqrt{k}\Z\right).
  \]
  Since $\eta\in G_\theta$, it follows that
  \begin{equation}
    \label{eqn:divides}
    z \;\mid\; c^2 s.
  \end{equation}
  Now consider the group $\Q + G_\theta$.  We clearly have
  \[
    \Q + G_\theta = \Q + \frac{b}{c}\sqrt{k}\left(\Z + \frac{2a}{c}\Z\right) = \Q + \frac{b}{cd}\sqrt{k}\Z
  \]
  where $d$ is the denominator of $\frac{2a}{c}$ when written in lowest form.  Since $\Q + G_\theta = \Q + G_\eta$, we
  must have
  \begin{equation}
    \label{eqn:bcd}
    \frac{b}{cd} = \frac{y}{zw}
  \end{equation}
  where $w$ is the denominator of $2x / z$ when written in lowest form.

  Now by~\eqref{eqn:divides}, there can be only finitely many choices for $z$.  Moreover, $w$ must divide
  $z$, so there are only finitely many choices for $w$.  By equation~\eqref{eqn:bcd}, once $z$ and
  $w$ are fixed, there can only be finitely many choices for $y$.  Finally, we only have finitely many
  choices (mod $\Z$) for $x$, and hence we can have only finitely many choices for $\eta$ (mod $\Z$).

  Summarizing the above, we have
  \begin{theorem}
    \label{thm:isomorphism_criterion}
    Let $\theta$ be irrational.
    \begin{itemize}
      \item  If $\theta$ has degree $> 4$, then for all $\eta$, $B_\theta \cong B_\eta$ if and only if $\theta = \pm\eta
      \pmod{\Z}$.
      \item  Suppose $\theta$ has degree $4$, with minimal polynomial $p(x) = x^4 + \lambda_3 x^3 + \lambda_2 x^2 + \lambda_1
      x + \lambda_0$.  If, for some nonzero $k\in\Z$,
      \[
        \frac{1}{8}\lambda_3^3 - \frac{1}{2}\lambda_3\lambda_2 + \lambda_1 = \frac{1}{k^3},
      \]
      then $B_\theta \cong B_\eta$ if and only if $\eta = \pm \theta\pmod{\Z}$ or $\eta = \pm \zeta\pmod{\Z}$, where
      \[
        \zeta = \frac{1}{2} k \lambda_3 \theta + k\theta^2.
      \]
      Otherwise, $B_\theta \cong B_\eta$ if and only if $\eta = \pm \theta\pmod{\Z}$.
      \item  If $\theta$ has degree $2$ or $3$, then there are finitely many distinct values of $\eta$ (mod $\Z$) such that
      $B_\theta \cong B_\eta$.
    \end{itemize}
  \end{theorem}

  It is already apparent from Theorem~\ref{thm:isomorphism_criterion} that in the case where the degree of $\theta$ is less than or equal to four, that $B_\theta \cong B_\eta$ will happen more often than $A_\theta \simeq A_\eta$.  For the cases  $\deg{\theta}
  = 2,3$, we provide examples to illustrate that possibly even more can happen.
  \begin{example} \label{ex:deg3}
    Let $\theta$ be the real solution to $\theta^3 = \theta + 1$.  Then the equation~\eqref{eq:cubic} from case (d) in
    this case is
    \[
      a^3 - ab^2 - b^3 = 1
    \]
    whose integer solutions are $(1,0)$, $(0,-1)$, $(\pm 1, -1)$, and $(4,3)$.  These correspond to the values $\eta =
    \theta$, $-\theta^2$, $\pm \theta - \theta^2$, and $4\theta + 3\theta^2$.  Since $\deg{\theta} > 2$, these values
    are all distinct mod $\Z$.  
  \end{example}

  Finally, we provide a similar example in the case where $\deg{\theta} = 2$.
  \begin{example} \label{ex:deg2}
    Let $\theta = \frac{1 + \sqrt{2}}{3}$ and $\eta = \frac{1 + 2\sqrt{2}}{3}$.  Then (in the notation of case (d)) we
    have
    \[
      G_\theta = \frac{1}{3}\Z + \frac{1}{9}\sqrt{2}\Z = G_\eta.
    \]
    Therefore, $B_\theta\cong B_\eta.$ On the other hand, $\theta$ is not equal to $\pm \eta + n$ for any integer $n$.  
  \end{example}

\section{A C*-algebra generated by a faithful irreducible representation of a finitely generated torsion free nilpotent group that is not an A$\T$ algebra} \label{sec:blah}
So far all of the algebras we considered turned out to be A$\T$ algebras.  By Chris Phillips' theorem \cite{Phillips06} if $G$ is any finitely generated two-step nilpotent group and $\pi$ a faithful irreducible representation of $G$, then $C^*(\pi(G))$ is an A$\T$ algebra. Since a natural conjecture forms from the previous sentences we would like to point out that it is very easy to produce irreducible representations of finitely generated nilpotent groups (of nilpotency step necessarily larger than 2) that are not A$\T$ algebras.

To this end, let $\theta$ be irrational and $u$ and $v$ be generators of $A_\theta$ satisfying the commutation relation. Consider the automorphism $\beta$ of $A_\theta$ defined by $\beta(u)=u$ and $\beta(v)= u^2v.$ It is then clear from the Pimsner-Voiculescu six-term exact sequence that $K_1(A_\theta\rtimes_\beta \Z)$ contains an element of order 2.  Since all A$\T$ algebras have torsion free $K_1$, it follows that $A_\theta\rtimes_\beta \Z$ is not an A$\T$ algebra. We claim that $A_\theta \rtimes_\beta\Z$ is isomorphic to the C*-algebra generated by an irreducible representation of a 3-step nilpotent group.

Indeed, let $\HH_3\leq GL(3,\Z)$ be the Heisenberg group with generators $a,b.$  Then $\beta(a)=a$ and $\beta(b)=a^2b$ defines an automorphism of $\HH_3.$    Notice that $\beta$ fixes $Z(\HH_3).$  Moreover the induced action of $\beta$ on $\HH_3/Z(\HH_3)\cong \Z^2$ is given by the unipotent matrix $\left(\begin{array}{ll} 1 & 2\\ 0 & 1\\ \end{array} \right),$ showing that $\HH_3\rtimes_\beta\Z$ is a three step nilpotent group. Let $\pi_\theta$ be the representation of $\HH_3\rtimes_\beta\Z$ induced from $\theta\in \T=\widehat{Z(\HH_3)}.$ One checks fairly easily that $C^*(\pi_\theta(\HH_3\rtimes_\beta\Z))\cong A_\theta\rtimes_\beta\Z.$
\begin{remark} The above example can clearly be generalized (with minimal effort) in a variety of ways to produce all kinds of finitely generated $K_1$ groups.
\end{remark}
 
\section*{Acknowledgements} C. E. thanks Marius Dadarlat and Andrew Toms for several informative conversations about ordered K-theory and also to Marius for suggesting the method of proof for Lemma \ref{lem:groupK}.

\bibliographystyle{plain}

\end{document}